\documentclass[11pt]{amsart}
\usepackage{amsmath}
\usepackage{amssymb}
\usepackage{amsthm}
\usepackage{amscd}
\usepackage{enumerate}
\usepackage{verbatim}
\usepackage[all]{xy}
\usepackage{mathrsfs}

\theoremstyle{plain}
\newtheorem{thm}{Theorem}[section]
\newtheorem{prop}[thm]{Proposition}
\newtheorem{lem}[thm]{Lemma}
\newtheorem{cor}[thm]{Corollary}

\theoremstyle{definition}

\newtheorem{rmk}[thm]{Remark}

\newcommand{\Hom}{\mathrm{Hom}}

\newcommand{\Lie}{\mathrm{Lie}}

\newcommand{\Ker}{\mathrm{Ker}}

\newcommand{\Img}{\mathrm{Im}}
\newcommand{\prjt}{\mathrm{pr}}
\newcommand{\Fil}{\mathrm{Fil}}
\newcommand{\Spf}{\mathrm{Spf}}

\newcommand{\Spec}{\mathrm{Spec}}

\newcommand{\Span}{\mathrm{Span}}

\newcommand{\rig}{\mathrm{rig}}
\newcommand{\HT}{\mathrm{HT}}
\newcommand{\Hdg}{\mathrm{Hdg}}

\newcommand{\et}{\mathrm{et}}
\newcommand{\ord}{\mathrm{ord}}

\newcommand{\sExt}{\mathscr{E}\!{\it xt}}

\newcommand{\Kbar}{\bar{K}}

\newcommand{\okbar}{\mathcal{O}_{\bar{K}}}

\newcommand{\Acrys}{A_{\mathrm{crys}}}

\newcommand{\okey}{\mathcal{O}_K}
\newcommand{\oc}{\mathcal{O}_{\mathbb{C}}}
\newcommand{\oel}{\mathcal{O}_L}

\newcommand{\cC}{\mathcal{C}}
\newcommand{\cD}{\mathcal{D}}
\newcommand{\cE}{\mathcal{E}}
\newcommand{\cF}{\mathcal{F}}
\newcommand{\cG}{\mathcal{G}}
\newcommand{\cH}{\mathcal{H}}

\newcommand{\cJ}{\mathcal{J}}

\newcommand{\cM}{\mathcal{M}}

\newcommand{\cO}{\mathcal{O}}

\newcommand{\cT}{\mathcal{T}}

\newcommand{\Tcrys}{T_{\mathrm{crys}}^{*}}

\newcommand{\TSG}{T^{*}_{\SG}}

\newcommand{\tokbar}{\tilde{\mathcal{O}}_{\Kbar}}
\newcommand{\tokey}{\tilde{\mathcal{O}}_{K}}

\newcommand{\okbari}{\mathcal{O}_{\Kbar,i}}

\newcommand{\CRYS}{\mathrm{CRYS}}

\newcommand{\Mod}{\mathrm{Mod}}

\newcommand{\SG}{\mathfrak{S}}

\newcommand{\SGf}{\mathfrak{F}}
\newcommand{\SGm}{\mathfrak{M}}
\newcommand{\SGn}{\mathfrak{N}}
\newcommand{\SGl}{\mathfrak{L}}

\newcommand{\ModSGf}{\mathrm{Mod}_{/\mathfrak{S}_1}^{1,\varphi}}
\newcommand{\ModSGfinf}{\mathrm{Mod}_{/\mathfrak{S}_{\infty}}^{1,\varphi}}
\newcommand{\ModSGffr}{\mathrm{Mod}_{/\mathfrak{S}}^{1,\varphi}}

\newcommand{\ModSGfV}{\mathrm{Mod}_{/\mathfrak{S}_1}^{1,\varphi,V}}
\newcommand{\ModSGfinfV}{\mathrm{Mod}_{/\mathfrak{S}_{\infty}}^{1,\varphi,V}}
\newcommand{\ModSGffrV}{\mathrm{Mod}_{/\mathfrak{S}}^{1,\varphi,V}}

\newcommand{\ModSf}{\mathrm{Mod}_{/S_1}^{1,\varphi}}
\newcommand{\ModSfinf}{\mathrm{Mod}_{/S_\infty}^{1,\varphi}}
\newcommand{\ModSffr}{\mathrm{Mod}_{/S}^{1,\varphi}}
\newcommand{\ModSffrV}{\mathrm{Mod}_{/S}^{1,\varphi,V}}

\newcommand{\upi}{\underline{\pi}}

\newcommand{\bC}{\mathbb{C}}
\newcommand{\bD}{\mathbb{D}}

\newcommand{\bZ}{\mathbb{Z}}
\newcommand{\bQ}{\mathbb{Q}}
\newcommand{\bF}{\mathbb{F}}

\newcommand{\sS}{\mathscr{S}}

\newcommand{\frA}{\mathfrak{A}}
\newcommand{\frG}{\mathfrak{G}}

\newcommand{\frX}{\mathfrak{X}}

\makeatletter
\renewcommand{\p@enumii}{}
\makeatother

\begin{document}

\title[Canonical subgroups for $p=2$]{Canonical subgroups via Breuil-Kisin modules for $p=2$}
\author{Shin Hattori}
\date{\today}
\email{shin-h@math.kyushu-u.ac.jp}
\address{Faculty of Mathematics, Kyushu University}
\thanks{Supported by Grant-in-Aid for Young Scientists B-23740025.}

\begin{abstract}
Let $p$ be a rational prime and $K/\bQ_p$ be an extension of complete discrete valuation fields. Let $\cG$ be a truncated Barsotti-Tate group of level $n$, height $h$ and dimension $d$ over $\okey$ with $0<d<h$. In this paper, we prove the existence of higher canonical subgroups for $\cG$ with standard properties if the Hodge height of $\cG$ is less than $1/(p^{n-2}(p+1))$, including the case of $p=2$.
\end{abstract}

\maketitle

\section{Introduction}\label{intro}

Let $p$ be a rational prime and $K/\bQ_p$ be an extension of complete discrete valuation fields. Let $k$ be its residue field, $\pi$ be its uniformizer, $e$ be its absolute ramification index, $\Kbar$ be its algebraic closure and $v_p$ be its valuation extended to $\Kbar$ and normalized as $v_p(p)=1$. We let $\bC$ denote the completion of $\Kbar$. For any valuation field $F$ (of height one) with valuation $v_F$ and valuation ring $\cO_F$, put
$m_F^{\geq i}=\{x \in F\mid v_F(x)\geq i\}$
and $\cO_{F,i}=\cO_F/m_F^{\geq i}$ for any positive real number $i$. We also put $\tokey=\mathcal{O}_{K,1}$, $\tokbar=\mathcal{O}_{\Kbar,1}$ and $\sS_i=\Spec(\cO_{K,i})$. 

One of the key ingredients of the theory of $p$-adic Siegel modular forms is the existence theorem of canonical subgroups. Let $\frX$ be the formal completion of the Siegel modular variety of genus $g$ and level prime to $p$ over the Witt ring $W(k)$ along the special fiber, $X$ be its Raynaud generic fiber, $X^\ord$ be its ordinary locus considered as an admissible open subset of $X$ and $\frA$ be the universal abelian scheme over $\frX$. Consider the unit component $\frA[p^n]^0$ of the $p^n$-torsion of $\frA$ and its Raynaud generic fiber $(\frA[p^n]^0)^\rig$. The restriction $(\frA[p^n]^0)^\rig|_{X^\ord}$ is etale locally isomorphic to the constant group $(\bZ/p^n\bZ)^g$ and it is a lift of the kernel of the $n$-th iterated Frobenius of the special fiber of $\frA$. Then the theorem asserts that this subgroup can be extended to a subgroup $C_n$ with the same properties over a larger admissible open subset of $X$ containing $X^\ord$. In \cite{Ha_cansub}, the author proved the existence of such a subgroup over the locus of the Hodge height less than $1/(2p^{n-1})$ for $p\geq 3$. The aim of this paper is to generalize the result to the case of $p=2$.

To state the main theorem, we introduce some notation. For any finite flat (commutative) group scheme $\cG$ ({\it resp.} Barsotti-Tate group $\Gamma$) over $\okey$, we let $\omega_\cG$ ({\it resp.} $\omega_\Gamma$) denote its module of invariant differentials. Put $\deg(\cG)=\sum_i v_p(a_i)$ by writing $\omega_\cG\simeq \oplus_i\okey/(a_i)$. For any positive rational number $i$, the Hodge-Tate map for a finite flat group scheme $\cG$ over $\okey$ killed by $p^n$ is defined to be the natural homomorphism
\[
\HT_i:\cG(\okbar)\simeq\Hom(\cG^\vee\times\Spec(\okbar),\mu_{p^n}\times\Spec(\okbar)) \to \omega_{\cG^\vee}\otimes \okbari
\]
defined by $g\mapsto g^*(dT/T)$, where $\vee$ means the Cartier dual and $\mu_{p^n}=\Spec(\okey[T]/(T^{p^n}-1))$ is the group scheme of $p^n$-th roots of unity. We normalize the upper and the lower ramification subgroups of $\cG$ to be adapted to the valuation $v_p$. Namely, writing the affine algebra of $\cG$ as 
\[
\okey[T_1,\ldots,T_r]/(f_1,\ldots,f_s)
\]
and an $r$-tuple $(x_1,\ldots,x_r)\in \okbar^r$ as $\underline{x}$, we put
\begin{align*}
\cG^j(\okbar)&=\cG(\okbar)\cap \{\underline{x}\in \okbar^r\mid v_p(f_l(\underline{x}))\geq j\text{ for any $l$}\}^0,\\
\cG_i(\okbar)&=\Ker(\cG(\okbar)\to \cG(\okbari)),
\end{align*}
where $(-)^0$ in the first equality means the geometric connected component as an affinoid variety over $K$ containing the zero section (see \cite[Section 2]{AM}). We also put $\cG^{j+}(\okbar)=\cup_{j'>j}\cG^{j'}(\okbar)$ for any non-negative rational number $j$. The scheme-theoretic closure of $\cG^{j}(\okbar)$ in $\cG$ is denoted by $\cG^j$ and define $\cG^{j+}$ and $\cG_i$ similarly. Finally, for any truncated Barsotti-Tate group $\cG$ (\cite{Il}) of level $n$, height $h$ and dimension $d$ over $\okey$ with $d<h$, we define the Hodge height $\Hdg(\cG)$ to be the truncated valuation $v_p(\det(V))\in [0,1]$ of the determinant of the natural action of the Verschiebung $V$ of the group scheme $\cG[p]^\vee\times\Spec(\tokey)$ on the free $\tokey$-module of finite rank $\Lie(\cG[p]^\vee\times\Spec(\tokey))$. Then our main theorem is the following, which is proved in \cite{Ha_cansub} except the case of $p=2$.

\begin{thm}\label{mainZ}
Let $p$ be a rational prime and $K/\bQ_p$ be an extension of complete discrete valuation fields. Let $\cG$ be a truncated Barsotti-Tate group of level $n$, height $h$ and dimension $d$ over $\okey$ with $0<d<h$ and Hodge height $w=\Hdg(\cG)$. 
\begin{enumerate}
\item If $w<1/(p^{n-2}(p+1))$, then there exists a finite flat closed subgroup scheme $\cC_n$ of $\cG$ of order $p^{nd}$ over $\okey$, which we call  the level $n$ canonical subgroup of $\cG$, such that $\cC_n\times \sS_{1-p^{n-1}w}$ coincides with the kernel of the $n$-th iterated Frobenius homomorphism $F^n$ of $\cG\times \sS_{1-p^{n-1}w}$. Moreover, the group scheme $\cC_n$ has the following properties:
\begin{enumerate}[(a)]
\item\label{mainZ-deg} $\deg(\cG/\cC_n)=w(p^n-1)/(p-1)$.
\item\label{mainZ-iso} Put $\cC_n'$ to be the level $n$ canonical subgroup of $\cG^\vee$. Then we have the equality of subgroup schemes $\cC_n'=(\cG/\cC_n)^\vee$, or equivalently $\cC_n(\okbar)=\cC_n'(\okbar)^\bot$, where $\bot$ means the orthogonal subgroup with respect to the Cartier pairing.
\item\label{mainZ-ram1} If $n=1$, then $\cC_1=\cG_{(1-w)/(p-1)}=\cG^{pw/(p-1)+}$.
\end{enumerate}

\item If $w<(p-1)/(p^n-1)$, then the subgroup scheme $\cC_n$ also satisfies the following: 
\begin{enumerate}[(a)]
\setcounter{enumii}{3}
\item\label{mainZ-free}   the group $\cC_n(\okbar)$ is isomorphic to $(\bZ/p^n\bZ)^d$.
\item\label{mainZ-sub} The scheme-theoretic closure of $\cC_n(\okbar)[p^i]$ in $\cC_n$ coincides with the subgroup scheme $\cC_i$ of $\cG[p^i]$ for $1\leq i\leq n-1$.
\end{enumerate}

\item\label{mainZ-HT} If $w<(p-1)/p^n$, then the subgroup  $\cC_n(\okbar)$ coincides with the kernel of the Hodge-Tate map $\HT_{n-w(p^n-1)/(p-1)}$.

\item\label{mainZ-ram} If $w<1/(2p^{n-1})$, then the subgroup scheme $\cC_n$ coincides with the upper ramification subgroup scheme $\cG^{j+}$ for any $j$ satisfying
\[
p w(p^n-1)/(p-1)^2 \leq j <p(1-w)/(p-1).
\]
\end{enumerate}
\end{thm}

We also show the uniqueness of the canonical subgroup $\cC_n$ for $w<p(p-1)/(p^{n+1}-1)$ (Proposition \ref{uniqueFrnZ}). From Theorem \ref{mainZ}, we can show the following corollary just as in the proof of \cite[Corollary 1.2]{Ha_cansub}.
\begin{cor}\label{cansubfamily}
Let $K/\bQ_p$ be an extension of complete discrete valuation fields. Let $\frX$ be an admissible formal scheme over $\Spf(\okey)$ which is quasi-compact and $\frG$ be a truncated Barsotti-Tate group of level $n$ over $\frX$ of constant height $h$ and dimension $d$ with $0<d<h$. We let $X$ and $G$ denote the Raynaud generic fibers of the formal schemes $\frX$ and $\frG$, respectively. For a finite extension $L/K$ and $x\in X(L)$, we put $\frG_x=\frG\times_{\frX,x}\Spf(\oel)$, where we let $x$ also denote the map $\Spf(\oel)\to \frX$ obtained from $x$ by taking the scheme-theoretic closure and the normalization. For a non-negative rational number $r$, let $X(r)$ be the admissible open subset of $X$ defined by 
\[
X(r)(\Kbar)=\{x\in X(\Kbar)\mid \Hdg(\frG_x)<r\}.
\]
Put $r_1=p/(p+1)$ and $r_n=1/(2p^{n-1})$ for $n\geq 2$. 

Then there exists an admissible open subgroup $C_n$ of $G|_{X(r_n)}$ such that, etale locally on $X(r_n)$, the rigid-analytic group $C_n$ is isomorphic to the constant group $(\bZ/p^n\bZ)^d$ and, for any finite extension $L/K$ and $x\in X(L)$, the fiber $(C_n)_x$ coincides with the generic fiber of the level $n$ canonical subgroup of $\frG_x$.
\end{cor}

The basic strategy of the proof of the main theorem is the same as in \cite{Ha_cansub}: we construct the level one canonical subgroup by lifting the conjugate Hodge filtration of a reduction of $\cG$ to the subobject of the associated Breuil-Kisin module of $\cG$. Since the canonical subgroup is required to be a lift of the Frobenius kernel of a reduction of $\cG$, it should be connected. Thus what we need is just a classification of connected finite flat group schemes allowing the case of $p=2$, which is due to Kisin (\cite{Ki_2}). Then the properties of our canonical subgroup follow easily from the construction, except the coincidence with ramification subgroups. 

In \cite{Ha_cansub}, to show that the canonical subgroups appear in the ramification filtrations of $\cG$, we use the ramification correspondence theorem (\cite[Theorem 1.1]{Ha_ramcorr}), which is shown only for $p\geq 3$. Instead, we prove the compatibility of the canonical subgroups with any finite base extension and reduce ourselves to the case of $\cG(\okbar)=\cG(\okey)$, where we can easily show that the lower ramification subgroups of $\cG$ are computed on the side of equal characteristic.



\section{Classification of unipotent finite flat group schemes}\label{ReviewZK}

In this section, we assume that the residue field $k$ of $K$ is perfect. For $p\geq 3$, we have a classification theory of Barsotti-Tate groups and finite flat group schemes over $\okey$ due to Breuil (\cite{Br}, \cite{Br_AZ}) and Kisin (\cite{Ki_Fcrys}, \cite{Ki_BT}) in terms of so-called Breuil-Kisin modules. Kisin (\cite{Ki_2}) also extended this classification to the case of $p=2$ and where groups are connected, using Zink's classification of formal Barsotti-Tate groups (\cite{Zi_W}, \cite{Zi}). In this section, we briefly recall this result of Kisin. Since we adopt a contravariant notation contrary to his, what we describe here is a classification of unipotent Barsotti-Tate groups and unipotent finite flat group schemes.

Let $W=W(k)$ be the Witt ring of $k$ and $\varphi$ be its natural Frobenius endomorphism which lifts the $p$-th power map of $k$. Natural $\varphi$-semilinear Frobenius endomorphisms of various $W$-algebras are denoted also by $\varphi$. Let $E(u)\in W[u]$ be the Eisenstein polynomial of $\pi$ over $W$. Let us fix once and for all a system $\{\pi_n\}_{n\geq 0}$ of $p$-power roots of $\pi$ in $\Kbar$ with $\pi_0=\pi$ and $\pi_{n+1}^p=\pi_n$. Put $K_\infty=\cup_{n\geq 0} K(\pi_n)$, $\SG=W[[u]]$ and $\SG_1=k[[u]]$. We write the $\varphi$-semilinear continuous ring endomorphisms of the latter two rings defined by $u\mapsto u^p$ also as $\varphi$. Then a Kisin module is an $\SG$-module $\SGm$ endowed with a $\varphi$-semilinear map $\varphi_\SGm:\SGm\to \SGm$. We write $\varphi_\SGm$ also as $\varphi$ if no confusion may occur. We follow the notation of \cite[Subsection 2.1]{Ha_ramcorr} and \cite[Subsection 2.1]{Ha_cansub}. In particular, we have categories $\ModSGffr$, $\ModSGf$, $\ModSGfinf$ of Kisin modules of $E$-height $\leq 1$ and a category $\Mod_{/B}^{1,\varphi}$ for any $k[[u]]$-algebra $B$.

Let $\SGm$ be an object of the category $\ModSGfinf$ and put $\varphi^{*}\SGm=\SG\otimes_{\varphi,\SG}\SGm$. Then the map $1\otimes\varphi:\varphi^{*}\SGm\to \SGm$ is injective and we have a unique map $\psi_\SGm: \SGm\to \varphi^{*}\SGm$ satisfying $(1\otimes\varphi)\circ \psi_\SGm=E(u)$. We say $\SGm$ is $V$-nilpotent if the composite 
\[
\varphi^{n-1*}(\psi_\SGm)\circ \varphi^{n-2*}(\psi_\SGm)\circ\cdots\circ\psi_\SGm: \SGm\to \varphi^{n*}\SGm
\]
factors through the submodule $(p,u)\varphi^{n*}\SGm$ for any sufficiently large $n$. Similarly, we say an object $\SGm$ of the category $\ModSGffr$ is topologically $V$-nilpotent if the same condition holds. The full subcategories of $V$-nilpotent ({\it resp.} topologically $V$-nilpotent) objects are denoted by $\ModSGfV$ and $\ModSGfinfV$ ({\it resp.} $\ModSGffrV$). Note that these notions are called connected and formal in \cite{Ki_2}, respectively.

Let $S$ be the $p$-adic completion of the divided power envelope of $W[u]$ with respect to the ideal $(E(u))$. The ring $S$ has a natural filtration $\Fil^1S$ induced by the divided power structure, a $\varphi$-semilinear Frobenius endomorphism denoted also by $\varphi$ and a $\varphi$-semilinear map $\varphi_1:\Fil^1S\to S$, since the inclusion $\varphi(\Fil^1S)\subseteq pS$ holds for any $p$. Then a Breuil module is an $S$-module $\cM$ endowed with an $S$-submodule $\Fil^1\cM$ containing $(\Fil^1S)\cM$ and a $\varphi$-semilinear map $\varphi_{1,\cM}: \Fil^1\cM\to \cM$ satisfying some compatibility conditions. The map $\varphi_{1,\cM}$ is also denoted by $\varphi_1$ if there is no risk of confusion. We also have categories of Breuil modules $\ModSffr$, $\ModSf$ and $\ModSfinf$ (for the definitions, see \cite[Subsection 2.1]{Ha_ramcorr}. Though there these are defined only for $p\geq 3$, the definitions are valid also for the case of $p=2$). For any object $\cM$ of these categories, we define a $\varphi$-semilinear map $\varphi_\cM:\cM\to \cM$ by $\varphi_\cM(x)=\varphi_1(E(u))^{-1}\varphi_1(E(u)x)$, which we abusively write as $\varphi$. 

Let $\cM$ be an object of the category $\ModSffr$ and put $\varphi^{*}\cM=S\otimes_{\varphi,S}\cM$. Then the map $1\otimes\varphi:\varphi^*\cM\to\cM$ is injective and we have a unique map $\psi_\cM:\cM\to \varphi^*\cM$ satisfying $(1\otimes\varphi)\circ \psi_\cM=p$. Then we say $\cM$ is topologically $V$-nilpotent if the composite
\[
\varphi^{n-1*}(\psi_\cM)\circ \varphi^{n-2*}(\psi_\cM)\circ\cdots\circ\psi_\cM: \cM\to \varphi^{n*}\cM
\]
factors through the submodule $(p,\Fil^1S)\varphi^{n*}\cM$ for any sufficiently large $n$. This notion is called $S$-window over $\okey$ in \cite{Ki_2} and \cite{Zi_W}. The full subcategory of topologically $V$-nilpotent objects is denoted by $\ModSffrV$. For any Kisin module $\SGm$, define a Breuil module $\cM_\SG(\SGm)=S\otimes_{\varphi,\SG}\SGm$ by putting
\begin{align*}
&\Fil^1\cM_\SG(\SGm)=\Ker(S\otimes_{\varphi,\SG}\SGm\overset{1\otimes\varphi}{\to} S/\Fil^1 S\otimes_\SG \SGm),\\
&\varphi_1:\Fil^1\cM_\SG(\SGm) \overset{1\otimes\varphi}{\to} \Fil^1 S\otimes_\SG \SGm \overset{\varphi_1\otimes 1}{\to}\cM_\SG(\SGm).
\end{align*}
This gives exact functors $\ModSGfinf \to \ModSfinf$ and $\ModSGffr\to \ModSffr$, which are both denoted by $\cM_\SG(-)$, and the latter induces a functor $\ModSGffrV\to \ModSffrV$ (\cite[Proposition 1.2.5]{Ki_2}).

We can associate Galois representations to Kisin and Breuil modules. Consider the ring $R=\varprojlim (\tokbar \gets \tokbar\gets \cdots)$, where the transition maps are $p$-th power maps. An element $r\in R$ is written as $r=(r_n)_{n\geq 0}$ with $r_n\in \tokbar$, and define $r^{(0)}\in \oc$ by $r^{(0)}=\lim_{n\to\infty} \hat{r}_n^{p^n}$, where $\hat{r}_n$ is a lift of $r_n$ in $\okbar$. Then the ring $R$ is a complete valuation ring of characteristic $p$ with its valuation defined by $v_R(r)=v_p(r^{(0)})$ whose fraction field is algebraically closed, and we put $m_R^{\geq i}=\{r\in R\mid v_R(r)\geq i\}$ and $R_i=R/m_R^{\geq i}$. We have a natural ring surjection $W(R)\to \oc$ which lifts the zeroth projection $\prjt_0: R\to \tokbar$. The ring $\Acrys$ is the $p$-adic completion of the divided power envelope of $W(R)$ with respect to the kernel of this surjection. Thus we have the induced surjection $\Acrys\to \oc$. Put $\upi=(\pi,\pi_1,\pi_2,\ldots)\in R$ and consider the rings $W(R)$ and $\Acrys$ as $\SG$-algebras by the map $u\mapsto [\upi]$. In particular, we consider the ring $k[[u]]$ as a subring of $R$ by the map $u\mapsto \upi$ and let $v_R$ also denote the induced valuation on the former ring, which satisfies $v_R(u)=1/e$.

For any objects $\SGm$ of the category $\ModSGffr$ and $\cM$ of $\ModSffr$, we associate to them $G_{K_\infty}$-modules
\begin{align*}
T^*_\SG(\SGm)&=\Hom_{\SG, \varphi}(\SGm, W(R)),\\
\Tcrys(\cM)&=\Hom_{S,\Fil^1,\varphi}(\cM,\Acrys)
\end{align*}
(\cite[Proposition B1.8.3]{F_PhiGamma} and \cite[Subsection 1.2.6]{Ki_2}). If the $\SG$-module $\SGm$ is free of rank $h$, then the $\bZ_p$-module $\TSG(\SGm)$ is also free of rank $h$ (\cite[Corollary 2.1.4]{Ki_Fcrys}). We also have a natural injection of $G_{K_\infty}$-modules $\TSG(\SGm)\to \Tcrys(\cM_\SG(\SGm))$ defined by $f\mapsto 1\otimes (\varphi\circ f)$ and this is a bijection if $\SGm$ is topologically $V$-nilpotent (\cite[Proposition 1.2.7]{Ki_2}). Similarly, for any object $\SGm$ of the category $\ModSGfinf$, we have the associated $G_{K_\infty}$-module
\[
T^*_\SG(\SGm)=\Hom_{\SG, \varphi}(\SGm, W(R)\otimes \bQ_p/\bZ_p).
\]

Put $\sS_n=\Spec(\cO_{K,n})$, $S_n=S/p^n S$ and $E_n=\Spec(S_n)$. Let us consider the big crystalline site $\CRYS(\sS_n/E_n)$ with the fppf topology and its topos $(\sS_n/E_n)_\CRYS$. For any Barsotti-Tate group $\Gamma$ over $\okey$, we have the contravariant Dieudonn\'{e} crystal $\bD^{*}(\Gamma\times\sS_n)=\sExt^1_{\sS_n/E_n}(\underline{\Gamma\times\sS_n},\cO_{\sS_n/E_n})$ (for the notation, see \cite{BBM}). We put
\[
\bD^{*}(\Gamma)(S\to \okey)=\varprojlim_n \bD^{*}(\Gamma\times\sS_n)(S_n\to \cO_{K,n}).
\]
This module is considered as an object $\Mod(\Gamma)$ of the category $\ModSffr$ with the natural $\varphi$-semilinear Frobenius map induced by the Frobenius of $\Gamma\times \sS_1$ and the filtration defined as the inverse image of the natural inclusion
\[
\omega_\Gamma \subseteq \varprojlim_n \bD^{*}(\Gamma\times\sS_n)(\cO_{K,n}\to \cO_{K,n}).
\]
The $\Acrys$-module
\[
\bD^{*}(\Gamma)(\Acrys\to \oc)=\varprojlim_n \bD^{*}(\Gamma\times\sS_n)(\Acrys/p^n\Acrys\to \cO_{\bC,n})
\]
also has a $\varphi$-semilinear Frobenius map and a filtration defined in the same way.
Similarly, for any finite flat group scheme $\cG$ over $\okey$, the $S$-module 
\[
\bD^{*}(\cG)(S\to \okey)=\varprojlim_n \bD^{*}(\cG\times \sS_n)(S_n\to \cO_{K,n})
\]
is endowed with a natural $\varphi$-semilinear Frobenius map which is induced by the Frobenius of the group scheme $\cG\times \sS_1$ and is also denoted by $\varphi$, and a filtration defined by the submodule
\[
\Fil^1\bD^{*}(\cG)(S\to \okey)=\varprojlim_n \sExt^1_{\sS_n/E_n}(\underline{\cG\times\sS_n},\cJ_{\sS_n/E_n})(S_n\to \cO_{K,n}),
\]
where $\cJ_{\sS_n/E_n}$ is the canonical divided power ideal sheaf of the structure sheaf $\cO_{\sS_n/E_n}$.

We say a Barsotti-Tate group or a finite locally free group scheme is unipotent if its Cartier dual is connected. We let $(\mathrm{BT}/\okey)^\mathrm{u}$ ({\it resp.} $(p\text{-}\mathrm{Gr}/\okey)^\mathrm{u}$) denote the category of unipotent Barsotti-Tate groups ({\it resp.} the category of unipotent finite flat group schemes killed by some $p$-power) over $\okey$.
If a Barsotti-Tate group $\Gamma$ over $\okey$ is unipotent, then the object $\Mod(\Gamma)$ is topologically $V$-nilpotent (\cite[Lemma 1.1.3]{Ki_2}). Moreover, we have the following classification theorem of unipotent Barsotti-Tate groups and unipotent finite flat group schemes, whose second assertion follows from the first assertion by an argument of taking a resolution (\cite[Subsection 1.3]{Ki_2}). 

\begin{thm}\label{KiZ}
\begin{enumerate}
\item\label{KiZ-1}(\cite{Ki_2}, Theorem 1.2.8)
There exists an anti-equivalence of exact categories
\[
\cG(-): \ModSGffrV \to (\mathrm{BT}/\okey)^{\mathrm{u}}
\]
with a natural isomorphism $\Mod(\cG(\SGm))\to \cM_\SG(\SGm)$. Moreover, we also have a natural isomorphism of $G_{K_\infty}$-modules 
\[
\varepsilon_\SGm: T_p(\cG(\SGm))\to \Tcrys(\cM_\SG(\SGm)).
\]
\item\label{KiZ-2}(\cite{Ki_2}, Theorem 1.3.9) There exists an anti-equivalence of exact categories
\[
\cG(-): \ModSGfinfV \to (p\text{-}\mathrm{Gr}/\okey)^{\mathrm{u}}
\]
with a natural isomorphism of $G_{K_\infty}$-modules 
\[
\varepsilon_\SGm: \cG(\SGm)(\okbar) \to \TSG(\SGm).
\]
\end{enumerate}
\end{thm}

On the other hand, for any $k[[u]]$-algebra $B$, we have an exact anti-equivalence $\cH(-)$ from the category $\Mod_{/B}^{1,\varphi}$ to a category of finite locally free group schemes over $B$ whose Verschiebung is the zero map (\cite[Th\'eor\`eme 7.4]{SGA3-7A}. See also \cite[Subsection 3.2]{Ha_ramcorr}). Moreover, for $B=k[[u]]$ and $\SGm\in \ModSGfV$, the anti-equivalences $\cG(-)$ and $\cH(-)$ are related by the natural isomorphism 
\[
\varepsilon_\SGm:\cG(\SGm)(\okbar) \to \TSG(\SGm)=\cH(\SGm)(R)
\] 
of Theorem \ref{KiZ} (\ref{KiZ-2}). We define the lower ramification subgroups of the group scheme $\cH(\SGm)$ to be adapted to the valuation $v_R$. Namely, we define
\[
\cH(\SGm)_i(R)=\Ker(\cH(\SGm)(R)\to\cH(\SGm)(R_i))
\]
for any positive rational number $i$. We also define $\deg(\cH(\SGm))$ by writing $\omega_{\cH(\SGm)}\simeq \oplus_i k[[u]]/(a_i)$ and putting $\deg(\cH(\SGm))=\sum_i v_R(a_i)$.

Let $\SGm$ be an object of the category $\ModSGffrV$. If we identify an element $g\in T_p(\cG(\SGm))$ with a homomorphism of Barsotti-Tate groups from $\bQ_p/\bZ_p$ to $\cG(\SGm)$ over $\oc$, then by the natural isomorphism $\Mod(\cG(\SGm))\to \cM_\SG(\SGm)$ the element $\varepsilon_\SGm(g)$ is identified with the induced map
\[
\bD^{*}(g): \bD^{*}(\cG(\SGm))(\Acrys\to \oc) \to \bD^{*}(\bQ_p/\bZ_p)(\Acrys\to \oc)=\Acrys.
\]
A similar argument as in the proof of \cite[Proposition 1.1.11]{Ki_BT} shows that for any exact sequence
\[
0 \to \SGn \to \SGn' \to \SGm \to 0
\]
of Kisin modules such that $\SGn$ and $\SGn'$ are objects of the category of $\ModSGffr$ and $\SGm$ is of $\ModSGfinf$, the functor $\cM_\SG(-)$ induces an exact sequence of Breuil modules
\[
0 \to \cM_\SG(\SGn) \to \cM_\SG(\SGn') \to \cM_\SG(\SGm) \to 0.
\]
This and \cite[Lemme 4.2.5 (ii)]{BBM} imply that, for any object $\SGm$ of the category $\ModSGfinfV$, there exists a natural isomorphism of $S$-modules 
\[
\bD^{*}(\cG(\SGm))(S\to \okey) \to \cM_\SG(\SGm)
\] 
which is compatible with $\Fil^1$ and $\varphi$.

Let us consider the $k$-algebra isomorphism $k[[u]]/(u^e)\to \tokey$ defined by $u\mapsto \pi$. Using this isomorphism, we identify the $k$-algebras of both sides. Then we can show the following lemma just as in the proof of \cite[Lemma 2.4]{Ha_cansub}.

\begin{lem}\label{HodgeZ}
\begin{enumerate}
\item \label{HodgeZ-BT}
Let $\cG$ be a unipotent truncated Barsotti-Tate group of level one over $\okey$ and $\SGm$ be the corresponding object of $\ModSGfV$ via the anti-equivalence $\cG(-)$. Then there exist natural isomorphisms of $\tokey$-modules
\[
\Fil^1\cM_\SG(\SGm)/(\Fil^1S)\cM_\SG(\SGm)\to \omega_{\cG},\quad \cM_\SG(\SGm)/\Fil^1\cM_\SG(\SGm)\to \Lie(\cG^\vee).
\]
\item \label{HodgeZ-p}
Let $\cG$ be a unipotent finite flat group scheme over $\okey$ killed by $p$ and $\SGm$ be the corresponding object of the category $\ModSGfV$ via the anti-equivalence $\cG(-)$. Then there exists a natural isomorphism of $\tokey$-modules $\Fil^1\cM_\SG(\SGm)/(\Fil^1S)\cM_\SG(\SGm)\to \omega_{\cG}$.
\item \label{HodgeZ-deg}
Let $\SGm$ be an object of the category $\ModSGfV$. Then we have the equalities
\[
\deg(\cG(\SGm))=\deg(\cH(\SGm))=v_R(\det(\varphi_\SGm)).
\]
\end{enumerate}
\end{lem}

Let $\cG$ be a unipotent truncated Barsotti-Tate group of level one, height $h$ and dimension $d$ over $\okey$ with $0<d<h$ and Hodge height $\Hdg(\cG)=w$. Let $\SGm$ be the object of the category $\ModSGfV$ corresponding to $\cG$ via the anti-equivalence $\cG(-)$. Put $\SGm_1=\SGm/u^e\SGm$ and 
\[
\Fil^1\SGm_1=\Img(1\otimes \varphi:\tokey\otimes_{\varphi, \tokey}\SGm_1\to \SGm_1).
\]
Then the modules $\SGm_1$ and $\Fil^1\SGm_1$ are naturally considered as objects of the category $\Mod_{/\tokey}^{1,\varphi}$. By Lemma \ref{HodgeZ} (\ref{HodgeZ-BT}), we can show that there exists a natural isomorphism of $\tokey$-modules $\Lie(\cG^\vee) \to \Fil^1\SGm_1$ as in \cite[Subsection 2.3]{Ha_cansub}. Thus the $\tokey$-module $\Fil^1\SGm_1$ is free of rank $h-d$. Moreover, we obtain an exact sequence of $\varphi$-modules over $\tokey$ 
\[
0\to \Fil^1\SGm_1 \to \SGm_1 \to \SGm_1/\Fil^1\SGm_1\to 0
\]
which splits as a sequence of $\tokey$-modules and the equality of truncated valuation $v_p(\det(\varphi_{\Fil^1\SGm_1}))=w$. We can also prove the following lemma as in the proof of \cite[Lemma 2.5]{Ha_cansub}.

\begin{lem}\label{HTmapZ}
Let $\cG$ be a unipotent truncated Barsotti-Tate group of level one over $\okey$ and $\SGm$ be the object of $\ModSGfV$ which corresponds to $\cG$ via the anti-equivalence $\cG(-)$. 
Then the composite
\[
\cG(\okbar)\overset{\varepsilon_\SGm}{\to} \cH(\SGm)(R)\to \Hom_{\tokey}(\Fil^1\SGm_1, R/m_R^{\geq i})\to \omega_{\cG^\vee}\otimes \cO_{\Kbar,i}
\]
coincides with the Hodge-Tate map $\HT_i$ for any $i\leq 1$.
\end{lem}

\begin{rmk}\label{KimRmk}
A similar classification for finite flat group schemes over $\okey$ via Breuil-Kisin modules allowing the case of $p=2$ and with non-trivial etale part is obtained independently by Kim (\cite{Kim_2}), Lau (\cite{Lau_2}) and Liu (\cite{Li_2}). By using \cite[Corollary 4.3]{Kim_2}, we can generalize Lemma \ref{HodgeZ} and Lemma \ref{HTmapZ} to this case.
\end{rmk}





\section{Canonical subgroups}\label{Cansub}

In this section, we prove Theorem \ref{mainZ}. The proof is a modification of the argument in \cite{Ha_cansub}, where we had to exclude the case of $p=2$. We begin with a consideration on a dual situation, as below. By this passage to the dual, we can replace the use of the congruence of the defining equations of $\cG(\SGm)$ and $\cH(\SGm)$ (\cite[Corollary 4.6]{Ha_ramcorr}) in \cite{Ha_cansub} to show the uniqueness of the level one canonical subgroup, by that of \cite[Proposition 1]{Fa} which is valid for any $p$. This uniqueness in turn replaces the use of the ramification correspondence theorem (\cite[Theorem 1.1]{Ha_ramcorr}) to show that the canonical subgroup coincides with a ramification subgroup.

Suppose that the residue field $k$ of $K$ is perfect. Let $\cG$ be a unipotent truncated Barsotti-Tate group of level one, height $h$ and dimension $d$ over $\okey$ with $0<d<h$ and Hodge height $\Hdg(\cG)=w$. Let $\SGm$ be the object of the category $\ModSGfV$ corresponding to $\cG$ via the anti-equivalence $\cG(-)$. Consider the objects $\SGm_1$ and $\Fil^1\SGm_1$ of the category $\mathrm{Mod}_{/\tokey}^{1,\varphi}$ as in the previous section. Then, a verbatim argument as in the proof of \cite[Lemma 3.3]{Ha_cansub} shows the following proposition.

\begin{prop}\label{conjHodge}
Let the notation be as above. For any finite flat closed subgroup scheme $\cD$ of $\cG$ over $\okey$, let $\SGl$ be the subobject of $\SGm$ in the category $\ModSGfV$ corresponding to the quotient $\cG/\cD$. Suppose $w<p/(p+1)$. Then there exists a unique $\cD$ satisfying $\SGl/u^{e(1-w)}\SGl=\Fil^1\SGm_1/u^{e(1-w)}\Fil^1\SGm_1$. Moreover, for this unique $\cD$, we have the equality $v_R(\det(\varphi_\SGl))=w$.
\end{prop}

We temporarily refer to the unique $\cD$ in Proposition \ref{conjHodge} as the conjugate Hodge subgroup of $\cG$, which will be shown to be equal to the canonical subgroup of $\cG$. 

\begin{lem}\label{cHFrobKer}
Let $\cG$ be as above and $\cD$ be the conjugate Hodge subgroup of $\cG$. Then $\cD$ is the unique finite flat closed subgroup scheme of $\cG$ over $\okey$ such that $(\cG/\cD)^\vee\times \sS_{1-w}$ coincides with the Frobenius kernel of $\cG^\vee \times \sS_{1-w}$. In particular, the construction of the conjugate Hodge subgroup is compatible with any finite extension of $K$.
\end{lem}
\begin{proof}
Put $i=1-w$. First let us show that the group scheme $(\cG/\cD)^\vee\times \sS_{i}$ coincides with the Frobenius kernel of $\cG^\vee\times \sS_{i}$. By comparing orders, it is enough to show that the group scheme $(\cG/\cD)^\vee\times \sS_{i}$ is contained in the Frobenius kernel of $\cG^\vee\times \sS_{i}$. By \cite[Proposition 1]{Fa}, it is equivalent to saying that the natural map 
\[
\omega_{\cG/\cD}\otimes \cO_{K,i} \to \omega_{\cG}\otimes \cO_{K,i}
\] 
is zero. By Lemma \ref{HodgeZ} (\ref{HodgeZ-p}), this map can be identified with the top horizontal arrow of the commutative diagram
\[
\xymatrix{
(\Fil^1\cM_\SG(\SGl)/(\Fil^1 S)\cM_\SG(\SGl)) \otimes \cO_{K,i} \ar[r]\ar[d] & (\Fil^1\cM_\SG(\SGm)/(\Fil^1 S)\cM_\SG(\SGl)) \otimes \cO_{K,i}\ar@{^{(}->}[d] \\
(\cM_\SG(\SGl)/(\Fil^1 S)\cM_\SG(\SGl)) \otimes \cO_{K,i} \ar@{^{(}->}[r] & (\cM_\SG(\SGm)/(\Fil^1 S)\cM_\SG(\SGm)) \otimes \cO_{K,i},
}
\]
where the bottom horizontal arrow and the right vertical arrow are injective. Let $\delta_1,\ldots, \delta_{h-d}$ be a basis of the $\SG$-module $\SGl$ and $D$ be the element of $M_{h-d}(k[[u]])$ satisfying
\[
\varphi(\delta_1,\ldots,\delta_{h-d})=(\delta_1,\ldots,\delta_{h-d})D.
\]
From the definition of the functor $\cM_\SG(-)$, we see that the module on the top left corner of the diagram is equal to the $\cO_{K,i}$-module
\[
\Span_{\tokey}((1\otimes \delta_1,\ldots,1\otimes \delta_{h-d})u^e D^{-1})\otimes \cO_{K,i}.
\]
The equality $v_R(\det(D))=w$ implies that the entries of the matrix $u^eD^{-1}$ are divisible by $u^{ei}$ and the left vertical arrow of the diagram is zero. Hence the assertion follows.

On the other hand, let $\cD'$ be a finite flat closed subgroup scheme of $\cG$ over $\okey$ such that $(\cG/\cD')^\vee\times \sS_{i}$ coincides with the Frobenius kernel of $\cG^\vee\times \sS_{i}$. Let $\SGl'$ be the subobject of $\SGm$ corresponding to the quotient $\cG/\cD'$. Since $\cG$ is a truncated Barsotti-Tate group of level one, the group scheme $\cD'\times\sS_{i}$ also coincides with the Frobenius kernel of $\cG\times \sS_{i}$. Note that the principal ideal $m_R^{\geq i}$ of the ring $R_{pi}$ has a unique divided power structure satisfying $\gamma_n(x)=0$ for any $x\in m_R^{\geq i}$ and $n\geq p$. We can consider the surjection $\prjt_0: R_{pi}\to \cO_{\Kbar, i}$ as a divided power thickening over the surjection $S\to \okey$. Evaluating the exact sequence
\[
0 \to \bD^{*}(\Img (F))\to \bD^{*}(\cG\times \sS_i)\to \bD^{*}(\Ker (F))\to 0
\]
on this divided power thickening, we obtain the equality 
\[
R_{pi}\otimes_S \cM_\SG(\SGl')=\Img(R_{pi}\otimes_{\varphi,S}\cM_\SG(\SGm) \overset{1\otimes\varphi}\to R_{pi}\otimes_S \cM_\SG(\SGm)).
\]
Since the $\tokey$-module $\Fil^1\SGm_1$ is a direct summand of $\SGm_1$, the $R_{pi}$-module $R_{pi}\otimes_{\varphi,\tokey}\Fil^1\SGm_1$ is a submodule of $R_{pi}\otimes_{\varphi,\tokey}\SGm_1$ and is equal to the image on the right-hand side of the above equality. Thus we obtain the equality
\[
R_{pi}\otimes_{\varphi, \cO_{K,i}}(\SGl'/u^{ei}\SGl') =  R_{pi}\otimes_{\varphi,\cO_{K,i}}(\Fil^1\SGm_1/u^{ei}\Fil^1\SGm_1)
\]
and the natural map 
\[
\SGl'/u^{ei}\SGl' \to (\SGm_1/u^{ei}\SGm_1)/(\Fil^1\SGm_1/u^{ei}\Fil^1\SGm_1)
\]
is zero after tensoring the injection $\varphi:\cO_{K,i}\to R_{pi}$. Since the $\cO_{K,i}$-modules of the both sides of this natural map is free, we obtain the inclusion $\SGl'/u^{ei}\SGl'\subseteq \Fil^1\SGm_1/u^{ei}\Fil^1\SGm_1$. Since the $\cO_{K,i}$-module $\SGl'/u^{ei}\SGl'$ is also a direct summand of $\SGm_1/u^{ei}\SGm_1$, the reverse inclusion follows similarly and the equality 
\[
\SGl'/u^{ei}\SGl'=\Fil^1\SGm_1/u^{ei}\Fil^1\SGm_1
\]
holds. Then the uniqueness assertion in Proposition \ref{conjHodge} implies the equality $\cD=\cD'$.

For the last assertion, let $L/K$ be a finite extension. Put $\cG_{\oel}=\cG\times \Spec(\oel)$ and similarly for $\cD_{\oel}$. The subgroup scheme $(\cG_{\oel}/\cD_{\oel})^\vee\times \sS_{1-w}$ also coincides with the Frobenius kernel of $(\cG_{\oel})^\vee\times \sS_{1-w}$. By the uniqueness we have just proved, the subgroup scheme $\cD_{\oel}$ coincides with the conjugate Hodge subgroup of $\cG_{\oel}$. This concludes the proof of the lemma.
\end{proof}

\begin{lem}\label{cHprop}
Let $\cG$ be as above and $\cD$ be the conjugate Hodge subgroup of $\cG$. 
\begin{enumerate}
\item\label{cH-ram1} $\cD=\cG_{(1-w)/(p-1)}$.
\item\label{cH-HT} If $w<(p-1)/p$, then the subgroup $\cD(\okbar)$ coincides with the kernel of the Hodge-Tate map
\[
\HT_b: \cG(\okbar)\to \omega_{\cG^\vee}\otimes \cO_{\Kbar,b}
\]
for any $b$ satisfying $w/(p-1) <b \leq 1-w$.
\item\label{cH-ram2} If $w<1/2$, then we have $\cD=\cG_b$ for any $b$ satisfying $w/(p-1)< b \leq (1-w)/(p-1)$. 
\end{enumerate}
\end{lem}
\begin{proof}
First we consider the assertion (\ref{cH-ram1}). It is enough to show the equality $\cD_{\oel}=(\cG_{\oel})_{(1-w)/(p-1)}$ for a finite extension $L/K$. By Lemma \ref{cHFrobKer}, the subgroup scheme $\cD_{\oel}$ is the conjugate Hodge subgroup of $\cG_{\oel}$. Thus we may assume $\cG(\okbar)=\cG(\okey)$. Let $\SGm$ and $\SGn$ be the objects of the category $\ModSGfV$ corresponding to $\cG$ and $\cD$, respectively. Then we can show the equality 
\[
\cH(\SGn)=\cH(\SGm)_{(1-w)/(p-1)}
\]
as in the proof of \cite[Theroem 3.1 (c)]{Ha_cansub}. Take $x\in \cG(\SGm)$. Let $\cG'$ be the scheme-theoretic closure in $\cG$ of the subgroup $\bF_px\subseteq \cG(\okbar)$ and $\SGm'$ be the quotient of $\SGm$ corresponding to $\cG'$. By the Oort-Tate classification (\cite{OT}), we have the following equivalences:
\begin{align*}
\cG(\SGm')_{(1-w)/(p-1)}=\cG(\SGm') &\Leftrightarrow v_R(\det(\varphi_{\SGm'}))\geq 1-w \\
&\Leftrightarrow \cH(\SGm')_{(1-w)/(p-1)}=\cH(\SGm').
\end{align*}

Note the commutative diagram
\[
\xymatrix{
\cG(\SGm')(\okbar) \ar[r]^{\varepsilon_{\SGm'}}_{\sim}\ar@{^{(}->}[d] & \cH(\SGm')(R) \ar@{^{(}->}[d]\\
\cG(\SGm)(\okbar) \ar[r]^{\varepsilon_{\SGm}}_{\sim} & \cH(\SGm)(R)\\
\cG(\SGn)(\okbar) \ar[r]^{\varepsilon_{\SGn}}_{\sim}\ar@{^{(}->}[u] & \cH(\SGn)(R),\ar@{^{(}->}[u]
}
\]
where the horizontal arrows are isomorphisms and the vertical arrows are injections. Then we have
\begin{align*}
x\in \cD(\okbar)=\cG(\SGn)(\okbar) & \Leftrightarrow \varepsilon_{\SGm}(x) \in \cH(\SGn)(R)=\cH(\SGm)_{(1-w)/(p-1)}(R)\\
& \Leftrightarrow \cH(\SGm')_{(1-w)/(p-1)}=\cH(\SGm')\\
& \Leftrightarrow \cG(\SGm')_{(1-w)/(p-1)}=\cG(\SGm')\\
& \Leftrightarrow x\in \cG(\SGm)_{(1-w)/(p-1)}(\okbar)
\end{align*}
and the assertion (\ref{cH-ram1}) follows. The assertions (\ref{cH-HT}) and (\ref{cH-ram2}) can be shown by a verbatim argument as in the proof of \cite[Theorem 3.1 (2), (3)]{Ha_cansub}. 
\end{proof}

\begin{rmk}
By using results of \cite{Kim_2} and Remark \ref{KimRmk}, we can drop the assumption that $\cG$ is unipotent from the arguments above, though we do no use this fact in what follows.
\end{rmk}

Now we proceed to prove the following theorem, which is the level one case of Theorem \ref{mainZ}.

\begin{thm}\label{cansubZ1}
Let the notation be as in Section \ref{intro}. Let $\cG$ be a truncated Barsotti-Tate group of level one, height $h$ and dimension $d$ over $\okey$ with $0<d<h$ and Hodge height $w=\Hdg(\cG)$. 
\begin{enumerate}
\item\label{cansubZ1-1} If $w<p/(p+1)$, then there exists a unique finite flat closed subgroup scheme $\cC$ of $\cG$ of order $p^d$ over $\okey$ such that $\cC\times \sS_{1-w}$ coincides with the kernel of the Frobenius homomorphism of $\cG\times \sS_{1-w}$. We refer to the subgroup scheme $\cC$ as the (level one) canonical subgroup of $\cG$. Moreover, the subgroup scheme $\cC$ has the following properties:
\begin{enumerate}
\item\label{cansubZ1-deg} $\deg(\cG/\cC)=w$.
\item\label{cansubZ1-iso} Let $\cC'$ be the canonical subgroup of $\cG^\vee$. Then we have the equality of subgroup schemes $\cC'=(\cG/\cC)^\vee$, or equivalently $\cC(\okbar)=\cC'(\okbar)^\bot$, where $\bot$ means the orthogonal subgroup with respect to the Cartier pairing.
\item\label{cansubZ1-ram1} $\cC=\cG_{(1-w)/(p-1)}=\cG^{pw/(p-1)+}$.
\end{enumerate}
\item\label{cansubZ1-HT} If $w<(p-1)/p$, then the subgroup $\cC(\okbar)$ coincides with the kernel of the Hodge-Tate map $\HT_b: \cG(\okbar)\to \omega_{\cG^\vee}\otimes \cO_{\Kbar, b}$ for any $b$ satisfying $w/(p-1) <b \leq 1-w$.
\item\label{cansubZ1-ram2} If $w<1/2$, then $\cC$ coincides both with the lower ramification subgroup scheme $\cG_b$ for any $b$ satisfying $w/(p-1)<b\leq (1-w)/(p-1)$ and the upper ramification subgroup scheme $\cG^{j+}$ for any $j$ satisfying $p w/(p-1)\leq j<p(1-w)/(p-1)$. 
\end{enumerate}
\end{thm}

\begin{proof}
By a base change argument as in the proof of \cite[Theorem 3.1]{Ha_cansub}, we may assume that the residue field $k$ of $K$ is perfect. Let $\cG^0$ ({\it resp.} $\cG^{\et}$) be the unit component ({\it resp.} the maximal etale quotient) of the group scheme $\cG$, and consider their Cartier duals $(\cG^0)^\vee$ and $(\cG^\et)^\vee$. These four group schemes are all truncated Barsotti-Tate groups of level one over $\okey$ and let $h_0$ be the height of $\cG^0$. Then $(\cG^0)^\vee$ is a unipotent truncated Barsotti-Tate group of level one, height $h_0$ and dimension $h_0-d$. If $h_0=d$, then the group scheme $\cG$ is ordinary (namely, $(\cG^0)^\vee$ is etale) and the assertions are clear. Thus we may assume $h_0>d$. 

Since $\Hdg(\cG)=\Hdg(\cG^\vee)$ and $\Lie(\cG)=\Lie(\cG^0)$, the truncated Barsotti-Tate group $(\cG^0)^\vee$ satisfies the assumption on the Hodge height in Proposition \ref{conjHodge}. Let $\cD$ be the conjugate Hodge subgroup of $(\cG^0)^\vee$, which is of order $p^{h_0-d}$. We define the canonical subgroup $\cC$ by $\cC=((\cG^0)^\vee/\cD)^\vee$, which is a finite flat closed subgroup scheme of $\cG^0$ of order $p^d$ over $\okey$. By Lemma \ref{cHFrobKer}, the group scheme $\cC\times \sS_{1-w}$ coincides with the Frobenius kernel of $\cG^0\times\sS_{1-w}$, which is equal to the Frobenius kernel of $\cG\times\sS_{1-w}$. If a finite flat closed subgroup scheme $\cE$ of $\cG$ has the reduction $\cE\times\sS_{1-w}$ equal to this Frobenius kernel, then $\cE$ is connected and Lemma \ref{cHFrobKer} implies $\cC=\cE$. Thus the uniqueness assertion of Theorem \ref{cansubZ1} (\ref{cansubZ1-1}) follows. 

By Lemma \ref{HodgeZ} (\ref{HodgeZ-deg}) and Proposition \ref{conjHodge}, we have $\deg((\cG^0)^\vee/\cD)=w$ and 
\begin{align*}
\deg(\cG/\cC)&=\deg(\cG^0/\cC)=\deg(\cD^\vee)=h_0-d-\deg(\cD)\\
&=h_0-d-\deg((\cG^0)^\vee)+\deg((\cG^0)^\vee/\cD)=w.
\end{align*}
Thus the part (\ref{cansubZ1-deg}) of the theorem follows. Cartier duality and the uniqueness of the canonical subgroup we have just proved imply the part (\ref{cansubZ1-iso}).

For the part (\ref{cansubZ1-ram1}), we insert here the following lemma due to the lack of references.

\begin{lem}\label{connupram}
Let $\cH$ be a finite flat group scheme over $\okey$ and $\cH^0$ be its unit component. Then we have $\cH^{j}=(\cH^0)^{j}$ for any positive rational number $j$.
\end{lem}
\begin{proof}
Replacing $K$ by a finite extension, we may assume that the maximal etale quotient $\cH^\et$ is a constant group scheme $\underline{M}$ for some abelian group $M$ and that we have an isomorphism of schemes over $\okey$
\[
\cH^0\times \underline{M}\to \cH
\]
which induces the natural isomorphism of group schemes $\cH^0\times\{0\}\to \cH^0$. Let $\cF^j$ be the functor of the set of geometric connected components of the $j$-th tubular neighborhood as in \cite[Section 2]{AM}. By \cite[Lemme 2.1.1]{AM}, the functor $\cF^j$ is compatible with products and we have a commutative diagram
\[
\xymatrix{
\cH^0(\okbar)\times \underline{M}(\okbar) \ar[r]\ar[d] & \cH(\okbar)\ar[d]\\
\cF^j(\cH^0)\times \cF^j(\underline{M}) \ar[r] &\cF^j(\cH),
}
\]
where the vertical arrows are homomorphisms and the horizontal arrows are bijections preserving zero elements. For $j>0$, the natural map $\underline{M}(\okbar)\to \cF^j(\underline{M})$ is an isomorphism and the kernel of the left vertical arrow is the subgroup $(\cH^0)^j(\okbar)\times\{0\}$. Thus the kernel of the right vertical arrow is the subgroup $(\cH^0)^j(\okbar)$ and the lemma follows.
\end{proof}

Now Lemma \ref{cHprop} (\ref{cH-ram1}) implies the equality $\cD=((\cG^0)^\vee)_b$ for $b=(1-w)/(p-1)$. 
By Lemma \ref{connupram} and a theorem of Tian and Fargues (\cite[Theorem 1.6]{Ti} or \cite[Proposition 6]{Fa}), we have the equalities $\cC=(\cG^0)^{j+}=\cG^{j+}$ for $j=p/(p-1)-p b=pw/(p-1)$. From this and the part (\ref{cansubZ1-iso}), we also obtain the equality $\cC=\cG_b$ and the part (\ref{cansubZ1-ram1}) follows. The part (\ref{cansubZ1-ram2}) can be shown similarly, by using Lemma \ref{cHprop} (\ref{cH-ram2}).

Finally we show the assertion (\ref{cansubZ1-HT}). By replacing $\cG$ by $\cG^\vee$, it is enough to show the assertion for the canonical subgroup $\cC'$ of $\cG^\vee$. Put $\cT=(\cG^\et)^\vee$. Let $\cD$ be the conjugate Hodge subgroup of $(\cG^0)^\vee$ as above and $\tilde{\cD}$ be the inverse image of $\cD$ by the natural epimorphism $\iota^\vee: \cG^\vee \to (\cG^0)^\vee$. We claim the equality $\cC'=\tilde{\cD}$. Indeed, by Lemma \ref{cHprop} (\ref{cH-ram1}), we have $\cD=((\cG^0)^\vee)_{(1-w)/(p-1)}$. By the part (\ref{cansubZ1-ram1}) of the theorem, this implies that $\cD$ is the canonical subgroup of $(\cG^0)^\vee$ and also coincides with $((\cG^0)^\vee)^{pw/(p-1)+}$. Since $\cC'=(\cG^\vee)^{pw/(p-1)+}$, the natural map $\iota^\vee$ induces the surjection $\cC'(\okbar)\to \cD(\okbar)$. In particular, the subgroup $\cC'(\okbar)$ is contained in $\tilde{\cD}(\okbar)$. On the other hand, since the group scheme $\cT$ is isomorphic to a direct sum of $\mu_p$ after a finite extension, we have the inclusions
\[
\cT=\cT_{1/(p-1)}\subseteq (\cG^\vee)_{1/(p-1)} \subseteq (\cG^\vee)_{(1-w)/(p-1)}=\cC',
\]
from which the claim follows.

Therefore, we have the commutative diagram with exact rows
\[
\xymatrix{
0 \ar[r] & \cT(\okbar) \ar[r]\ar@{=}[d] & \cC'(\okbar) \ar[r]\ar[d] & \cD(\okbar)\ar[r]\ar[d] & 0 \\
0 \ar[r] & \cT(\okbar) \ar[r] & \cG^\vee(\okbar) \ar[r]\ar[d] & (\cG^0)^\vee(\okbar)\ar[r]\ar[d] & 0 \\
 & & \omega_{\cG}\otimes \mathcal{O}_{\Kbar,b} \ar[r]^\sim & \omega_{\cG^0}\otimes \mathcal{O}_{\Kbar,b}, & 
}
\]
where the lowest vertical arrows are the Hodge-Tate maps $\HT_{b}$ and the lowest horizontal arrow is an isomorphism. By Lemma \ref{cHprop} (\ref{cH-HT}), the subgroup $\cD(\okbar)$ coincides with the kernel of the lowest right vertical arrow for any $b$ satisfying $w/(p-1)<b\leq 1-w$. This implies that the subgroup $\tilde{\cD}(\okbar)=\cC'(\okbar)$ is the kernel of the lowest left vertical arrow for any such $b$ and the assertion (\ref{cansubZ1-HT}) follows. This concludes the proof of Theorem \ref{cansubZ1}.
\end{proof}

Since the arguments in the proof of \cite[Theorem 1.1]{Ha_cansub} work verbatim also for $p=2$, Theorem \ref{mainZ} follows from Theorem \ref{cansubZ1}. 

Moreover, we can show the following proposition on anti-canonical subgroups, by modifying the proof of \cite[Proposition 4.3]{Ha_cansub}.

\begin{prop}\label{noncanisogZ}
Let $\cG$ be a truncated Barsotti-Tate group of level two, height $h$ and dimension $d$ over $\okey$ with $0<d<h$ and Hodge height $w=\Hdg(\cG)$. Suppose $w<1/2$ and let $\cC$ be the canonical subgroup of $\cG[p]$ as in Theorem \ref{cansubZ1}. Let $\cD$ be a finite flat closed subgroup scheme of $\cG[p]$ over $\okey$ such that the natural map $\cC(\okbar)\oplus \cD(\okbar)\to \cG[p](\okbar)$ is an isomorphism.
\begin{enumerate}
\item\label{noncanZ1}
The truncated Barsotti-Tate group $p^{-1}\cD/\cD$ of level one has Hodge height $\Hdg(p^{-1}\cD/\cD)=p^{-1}w$.
\item\label{noncanZ2}
The subgroup scheme $\cG[p]/\cD$ is the canonical subgroup of $p^{-1}\cD/\cD$.
\item\label{noncanZ3}
$\deg(\cD)=p^{-1}w$.
\end{enumerate}
\end{prop}
\begin{proof}
By a base change argument as before, we may assume that the residue field $k$ of $K$ is perfect. Note that the truncated Barsotti-Tate group $p^{-1}\cD/\cD$ of level one is also of height $h$ and dimension $d$. The natural homomorphism $\cC\to \cG[p]/\cD$ induces an isomorphism between the generic fibers of both sides. Since the group scheme $\cC$ is connected, the connected-etale sequence implies that the group scheme $\cG[p]/\cD$ is also connected. Now we claim that the group scheme $(\cG[p]/\cD)\times \sS_{1-w}$ is killed by the Frobenius. For this, let $\SGl$ and $\SGl'$ be the objects of the category $\ModSffrV$ corresponding to the unipotent finite flat group schemes $\cC^\vee$ and $(\cG[p]/\cD)^\vee$ via the anti-equivalence $\cG(-)$, respectively. By \cite[Corollary 2.2.2]{Li_FC}, the generic isomorphism $(\cG[p]/\cD)^\vee \to \cC^\vee$ corresponds to an injection $\SGl \to \SGl'$. Then the $\SG_1$-modules $\wedge^d \SGl$ and $\wedge^d\SGl'$ are free of rank one and this injection induces an injection $\wedge^d \SGl\to \wedge^d\SGl'$. Hence we obtain the inequality $v_R(\det \varphi_{\SGl'}) \leq v_R(\det \varphi_{\SGl})=w$. Since the group scheme $\cG[p]/\cD$ is a finite flat closed subgroup scheme of the truncated Barsotti-Tate group $(p^{-1}\cD/\cD)^0$ of level one over $\okey$, a similar argument to the proof of Lemma \ref{cHFrobKer} implies the claim. Then the proposition follows as in the proof of \cite[Proposition 4.3]{Ha_cansub}.
\end{proof}

We also have the following generalization of \cite[Proposition 4.4]{Ha_cansub} to the case of $p=2$.

\begin{prop}\label{uniqueFrnZ}
Let $\cG$ be a truncated Barsotti-Tate group of level $n$, height $h$ and dimension $d$ over $\okey$ with $0<d<h$ and Hodge height $w<p(p-1)/(p^{n+1}-1)$. Let $\cC_n$ be the level $n$ canonical subgroup of $\cG$, which is defined by Theorem \ref{mainZ}. Let $\cD_n$ be a finite flat closed subgroup scheme of $\cG$ over $\okey$ such that $\cD_n(\okbar)\simeq (\bZ/p^n\bZ)^d$ and the group scheme $\cD_n\times\sS_{1-p^{n-1}w}$ is killed by the $n$-th iterated Frobenius $F^n$. Then we have $\cC_n=\cD_n$.
\end{prop}
\begin{proof}
Suppose $n\geq 2$. Let $\cC_1$ be the scheme-theoretic closure of $\cC_n(\okbar)[p]$ in $\cC_n$ and define $\cD_1$ similarly. The group scheme $\cC_1$ coincides with the canonical subgroup of $\cG[p]$ by Theorem \ref{mainZ} (\ref{mainZ-sub}). We claim $\cC_1=\cD_1$, which implies the proposition by an induction as in the proof of \cite[Proposition 4.4]{Ha_cansub}. For this, by a base change argument as before, we may assume that the residue field $k$ of $K$ is perfect and $\cG(\okbar)=\cG(\okey)$. Suppose $\cC_1\neq \cD_1$. Then we can find a finite flat closed subgroup scheme $\cE$ of $\cG[p]$ such that $\cC_1(\okbar)\oplus \cE(\okbar)=\cG[p](\okbar)$ and $\cD_1(\okbar)\cap \cE(\okbar)\neq 0$. Let $\cF$ be the scheme-theoretic closure of the latter intersection in $\cG[p]$. As in the proof of Theorem \ref{cansubZ1} (\ref{cansubZ1-HT}), we see that the group scheme $\cC_1$ contains the multiplicative part of $\cG[p]$. Thus the quotient $\cG[p]/\cC$ is unipotent. Since the natural map $\cE\to \cG[p]/\cC$ is a generic isomorphism, the connected-etale sequence shows that $\cE$, and thus also $\cF$, are unipotent. Let $\SGf$ be the object of the category $\ModSGfV$ corresponding to $\cF$. Then Lemma \ref{HodgeZ} (\ref{HodgeZ-deg}) and Proposition \ref{noncanisogZ} (\ref{noncanZ3}) implies
\[
v_R(\det(\varphi_{\SGf}))=\deg(\cF)\leq \deg(\cE)=p^{-1}w.
\]

On the other hand, the group scheme $\cF\times \sS_{1-p^{n-1}w}$ is killed by the $n$-th iterated Frobenius. Put $i=1-p^{n-1}w$. Evaluating $\bD^{*}(\cF)$ on the divided power thickening $R_{pi}\to \cO_{\Kbar,i}$ as before, we see that the map
\[
1\otimes \varphi\otimes \varphi_{\SGf}: R_{pi}\otimes_{\varphi,R_{pi}} (R_{pi}\otimes_{\varphi,\SG}\SGf) \to R_{pi}\otimes_{\varphi,\SG}\SGf
\]
is the one induced by the Frobenius 
\[
\bD^{*}(F): \bD^{*}((\cF\times\sS_{i})^{(p)})\to \bD^{*}(\cF\times\sS_{i})
\] 
and that the composite of its pull-backs
\[
\varphi^{n*}(R_{pi}\otimes_{\varphi,\SG}\SGf) \to \varphi^{n-1*}(R_{pi}\otimes_{\varphi,\SG}\SGf) \to \cdots\to R_{pi}\otimes_{\varphi,\SG}\SGf
\]
is the map induced by the $n$-th iterated Frobenius. Taking the valuation of the determinant of this map, we obtain the inequalities 
\[
p(1-p^{n-1}w)\leq v_R(\det\varphi_\SGf)p(p^n-1)/(p-1)\leq w(p^n-1)/(p-1), 
\]
which contradict the assumption on $w$ and the equality $\cC_1=\cD_1$ follows.
\end{proof}




\begin{thebibliography}{9a}




\bibitem{AM}
A. Abbes and A. Mokrane: \emph{Sous-groupes canoniques et cycles \'evanescents $p$-adiques pour les vari\'et\'es ab\'eliennes}, Publ. Math. Inst. Hautes Etudes Sci. {\bf 99} (2004), 117--162. 





\bibitem{BBM}
P. Berthelot, L. Breen and W. Messing: \emph{Th\'eorie de Dieudonn\'e cristalline II}, Lecture Notes in Mathematics, {\bf 930}. Springer-Verlag, Berlin, 1982. x+261 pp.


\bibitem{Br}
C. Breuil: \emph{Groupes $p$-divisibles, groupes finis et modules filtr\'{e}s}, Ann. of Math. (2) {\bf 152} (2000), no. 2, 489--549. 

\bibitem{Br_AZ}
C. Breuil: \emph{Integral $p$-adic Hodge theory}, Algebraic geometry 2000, Azumino (Hotaka), 51--80, Adv. Stud. Pure Math., {\bf 36}, Math. Soc. Japan, Tokyo, 2002. 








\bibitem{Fa}
L. Fargues: \emph{La filtration canonique des points de torsion des groupes p-divisibles (avec la collaboration de Yichao Tian)}, Ann. Sci. Ecole Norm. Sup. (4) {\bf 44} (2011), no. 6, 905--961. 



\bibitem{F_PhiGamma}
J.-M. Fontaine: \emph{Repr\'esentations $p$-adiques des corps locaux I}, The Grothendieck Festschrift, Vol. II,  249--309, Progr. Math., {\bf 87}, Birkhauser Boston, Boston, MA, 1990.


\bibitem{SGA3-7A}
P. Gabriel: \emph{\'Etude infinit\'esimale des sch\'emas en groupes et groupes formels}, Sch\'emas en Groupes (S\'em. G\'eom\'etrie Alg\'ebrique, Inst. Hautes Etudes Sci., 1963/64), Fasc. 2b, Expos\'e 7a,  pp. 1-65+4, Inst. Hautes Etudes Sci., Paris.





\bibitem{Ha_ramcorr}
S. Hattori: \emph{Ramification correspondence of finite flat group schemes over equal and mixed characteristic local fields}, J. of Number Theory {\bf 132} (2012), no. 10, 2084--2102.

\bibitem{Ha_cansub}
S. Hattori: \emph{Canonical subgroups via Breuil-Kisin modules}, preprint, available at http://www2.math.kyushu-u.ac.jp/\~{}shin-h/

\bibitem{Il}
L. Illusie: \emph{D\'{e}formations de groupes de Barsotti-Tate (d'apr\`{e}s A. Grothendieck)}, Seminar on arithmetic bundles: the Mordell conjecture (Paris, 1983/84), Asterisque No. {\bf 127} (1985), 151--198. 


\bibitem{Kim_2}
W. Kim: \emph{The classification of $p$-divisible groups over $2$-adic discrete valuation rings}, Math. Res. Lett. {\bf 19} (2012), no. 1, 121--141.

\bibitem{Ki_Fcrys}
M. Kisin: \emph{Crystalline representations and $F$-crystals}, Algebraic geometry and number theory, 459--496, Progr. Math. {\bf 253}, Birkhauser Boston, Boston, MA, 2006.

\bibitem{Ki_BT}
M. Kisin: \emph{Moduli of finite flat group schemes and modularity}, Ann. of Math. (3) {\bf 170} (2009), 1085--1180.

\bibitem{Ki_2}
M. Kisin: \emph{Modularity of $2$-adic Barsotti-Tate representations}, Invent. Math. {\bf 178} (2009), no. 3, 587--634. 


\bibitem{Lau_2}
E. Lau: \emph{A relation between Dieudonn\'{e} displays and crystalline Dieudonn\'{e} theory}, arXiv:1006.2720v2 (2010). 

\bibitem{Li_FC}
T. Liu: \emph{Torsion $p$-adic Galois representations and a conjecture of Fontaine}, Ann. Sci. Ecole Norm. Sup. (4) {\bf 40} (2007), no. 4, 633--674. 

\bibitem{Li_2}
T. Liu: \emph{The correspondence between Barsotti-Tate groups and Kisin modules when $p=2$}, available at \verb|http://www.math.purdue.edu/~tongliu/research.html|.


\bibitem{OT}
J. Tate and F. Oort: \emph{Group schemes of prime order}, Ann. Sci. Ecole Norm. Sup. (4) {\bf 3} (1970), 1--21. 


\bibitem{Ti}
Y. Tian: \emph{Canonical subgroups of Barsotti-Tate groups}, Ann. of Math. (2) {\bf 172} (2010), no. 2, 955--988.




\bibitem{Zi_W}
T. Zink: \emph{Windows for displays of $p$-divisible groups}, Moduli of abelian varieties (Texel Island, 1999), 491--518,
Progr. Math., {\bf 195}, Birkhauser, Basel, 2001. 

\bibitem{Zi}
T. Zink: \emph{The display of a formal $p$-divisible group}, Cohomologies $p$-adiques et applications arithm\'{e}tiques I, Ast\'{e}tisque {\bf 278}, pp. 127--248, Soci\'{e}t\'{e} Math\'{e}matique de France, Paris, 2002.



\end{thebibliography}
\end{document}